\documentclass[12pt,a4paper]{article}
\usepackage{epsfig}
\usepackage[numbers]{natbib}
\usepackage{amsmath, amsthm, amssymb}
\usepackage{epsfig}
\usepackage[margin=3cm, top=3.5cm, bottom=3.5cm]{geometry}

\newcommand\blfootnote[1]{%
  \begingroup
  \renewcommand\thefootnote{}\footnote{#1}%
  \addtocounter{footnote}{-1}%
  \endgroup
}
\newtheorem{theorem}{Theorem}[section]
\newtheorem{lemma}{Lemma}[section]

\newtheorem{definition}{Definition}[section]

\numberwithin{equation}{section}
\everymath{\displaystyle}
\allowdisplaybreaks

\title{Nonlinear parabolic problem with time  fractional derivative}
\date{}
\author{
Nikolai Kutev\thanks{Institute of Mathematics and Informatics, Bulgarian Academy of Sciences, 1113, Sofia, Bulgaria}
 \and Tsviatko Rangelov
 \footnotemark[1]
}

\begin{document}

\maketitle
\blfootnote{Corresponding author: T. Rangelov, rangelov@math.bas.bg}

\begin{abstract}
\noindent
Time fractional parabolic problem for p-Laplacian with double singular Hardy-type potential is considered. Comparison principle and appriory estimates for the weak solutions are proved. Existence of  global weak solutions and finite-time blow-up  are investigated depending on the optimal Hardy constant.
\end{abstract}


{\bf Keywords} Fractional derivative on time,  p-Laplacian, Hardy inequality, Comparison principle, Global existence, Finite time blow-up.
\vspace{2pt}

{\bf Math. Subj. Class.}  35K92, 26A33, 35B44

\section{Introduction}
\label{sec1}
Let $D_t^\alpha$, $\alpha\in (0,1)$ be the Liouville-Caputo fractional derivative with respect to $t\in[0,\infty)$,
$$
D_t^\alpha h(t)=\frac{\partial}{\partial t}\frac{1}{\Gamma(1-\alpha)}\int_0^t(t-s)^{-\alpha}\left[h(s)-h(0)\right]ds,
$$
where $\Gamma$ is Gamma function, see \citet{Ki94} and \citet{KST06}.

Denote by $\Omega=\{x\in\mathbb{R}^n,   |x|\leq\varphi(x)\}$  a bounded star-shaped domain with respect to a small ball centered at the origin and $\varphi(x)\in C^{0,1}(\mathbb{R}^n)$ is a positive homogeneous function of 0-th  order.

We consider the parabolic problem with time fractional derivative
\begin{equation}
\label{eq1}
\left\{
\begin{array}{ll}
& D_t^\alpha u-\hbox{div}\left(|\nabla u|^{p-2}\nabla u\right)-\mu W(x)|u|^{p-2}u=0,\hbox{ in } Q=(0,T)\times\Omega,
\\
& u(t,x)=0,  \hbox{ for } (t,x)\in[0,T]\times\partial\Omega,
\\
& u(0,x)=u_0(x),\hbox{ for }x\in\Omega\subset\mathbb{R}^n, u_0(x)\in L^2(\Omega),
\end{array}\right.
\end{equation}
where $n>p$, $p>2$, $\mu>0$, $\Delta_p=\hbox{div}\left(|\nabla u|^{p-2}\nabla u\right)$ is p-Laplacian and the potential $W(x)$ is singular on the  boundary $\partial\Omega$  and at the origin
\begin{equation}
\label{eq3}
W(x)=|x|^{-p}\left[1-\left(\frac{|x|}{\varphi(x)}\right)^{\frac{n-p}{p}}\right]^{-p}.
\end{equation}

The theory of fractional evolutionary PDEs is intensively investigated in the last decades. Fractional time derivative PDEs, including Laplace, p-Laplace, porous medium, mean curvature, Kirchhoff operator and slow diffusion are studied in \citet{DVV19}, \citet{Ta13} and \citet{Uc00}. In \citet{ACV16}, \citet{DNA19}, \citet{DK21}, \citet{KY18}, \citet{LLY14}, \citet{Lu09}, \citet{VZ10}, \citet{VZ15} existence and uniqueness of the solutions is proved, while in \citet{ST23} optimal decay estimates are established.

Let us mention that evolutionary PDEs with fractional Laplace and p-Laplace operator are considered in \citet{MRT16}, \citet{FP16}, \citet{Bo22} where existence of local solutions, global solvability and final time blow-up by means of the potential well method are proved.

The applications of fractional calculus appear in the theory of mechanics of fluids, probability and statistics, dynamical processes in self-similar and and porous structures, viscoelasticity, electrochemistry of corrosion, optic and signal processing and so on, see \citet{KST06}. For example, in mathematical physics when describing dynamic processes in material with memory in the theory of heat conduction, models of anomalous diffusion, in the context of spatially disordered  systems, porous media, fractional media, sites of micro-molecular crowding, stock price movements ( \citet{BH00}, \citet{KNV07}, \citet{MK00}),  magneto-thermoelastic heat conduction (\citet{EK11}) and hydrodynamics (\citet{AK14}).

As for classical parabolic problems with singular potentials, they are models in molecular physics (\citet{Le67}), quantum cosmology (\citet{Ca04}), non-relativistic quantum mechanics (\citet{PV95}) and others.

Global existence or finite time blow-up is proved for inverse-square potential in \citet{BG84}, \citet{CM99} and for more general singularities at an interior point in \citet{CM99}, \citet{Fe97}, \citet{AT00} and \citet{AT02}. Singular potentials on the whole boundary are treated in \citet{BM92}, while double singular potentials at an interior point and on the boundary are investigated in \citet{KR22}.

The goal of the present paper is to extend the result for classical parabolic problems in the works mention above to time fractional evolutionary PDEs with p-Laplace operator and double singular potential. We prove finite time blow-up or global existence of the weak solutions to \eqref{eq1} and thresholds between them is the optimal Hardy constant, see Theorem 9.3 in \citet{KR22}.

Note that for   potential $W(x)$ the following Hardy inequality holds
\begin{equation}
\label{eq3a}
\int_\Omega|\nabla u|^pdx\geq\Lambda(n,p)\int_\Omega W(x)|u|^pdx,
\end{equation}
with optimal constant $\Lambda(n,p)=\left(\frac{n-p}{p}\right)^p$.

Let us construct  the truncated problem
\begin{equation}
\label{eq4}
\left\{
\begin{array}{ll}
&D_t^\alpha u_{N}-\Delta_p u_N-\mu W_N(x)|u_N|^{p-2}u_N=0,\hbox{ in } Q,
\\
&u_N(t,x)=0,  \hbox{ for } (t,x)\in[0,T]\times\partial\Omega, \ \
\\
&u_N(0,x)=u_0(x),\hbox{ for }x\in\Omega.
\end{array}
\right.
\end{equation}
for every positive integer $N\geq1$, where $W_N(x)=\min\{N, W(x)\}$.
\begin{definition}
\label{def1}
The function  $u(t,x)$ is a  weak solution in distributional sense of \eqref{eq1} if
$$
u(t,x)\in C^0\left([0,\tau];L^2(\Omega)\right)\cap L^p\left([0,\tau];W^{1,p}_{0}(\Omega)\right), \ \  \hbox{ for some } \tau>0,
$$
and is a limit of the solutions $u_N(t,x)$ of the truncated problem \eqref{eq4}, for $N\rightarrow\infty$.
\end{definition}
\begin{definition}
\label{def2}
The solution of \eqref{eq1} defined in $[0,\tau)$ does not blow-up at $\tau$ if there exists a constant $C_0$ such that
\begin{equation}
\label{eq4a}
\|u(t,x)\|_{L^2(Q_\tau)}\leq C_0, \quad \|u(t,x)\|_{W^{1,p}(Q_\tau)}\leq C_0 \hbox{ for } t\in [0,\tau), x\in \Omega,
\end{equation}
where $Q_\tau=(0,\tau)\times\Omega$.
\end{definition}
Our aim is to prove the following two theorems
\begin{theorem}
\label{th1-1}
Let  $\Omega\subset\mathbb{R}^n$, $n>p$, $p>2$, $0\in\Omega$, be a star-shaped domain with respect to a small ball centered at the origin,  and $\mu<\Lambda(n,p)$. If  $u_0(x)\in W_0^{1,p}(\Omega)$ then every weak solution of the problem  problem \eqref{eq1} defined in the maximal existence interval $(0, T_m)$ satisfies \eqref{eq4a} for every $\tau\in(0,T_m)$ and does not blow-up at $T_m$.
\end{theorem}
\begin{theorem}
\label{th1-2}
Suppose $\mu>\Lambda(n,p)$, $p>2$, $n>p$, $u_0(x)>0$ in $\bar{\Omega}$, $u_0(x)\in L^\infty(\Omega)\cap W_0^{1,p}(\Omega)$. Then problem \eqref{eq1} has no global in time weak solution because it blows-up for a finite time.
\end{theorem}

The plan of the paper is the following. In Sect. \ref{sec2} we consider fractional ordinary differential equation (FODE) and prove comparison principle (Theorem \ref{th3-1}) for Eq. \eqref{eq3-1}.  Then, for initial-value problem of special type  FODE \eqref{eqq1} and using the comparison principle (Theorem \ref{th1}, Lemma 2.1 in \citet{LV08}) we prove Theorem \ref{th2-3} - the existence of blow-up solution of the problem \eqref{eqq1}. Sect. \ref{sec2} is essential for the proof of the main results and has its own interest. In Sect. \ref{sec3-2} we establish a comparison principle (Theorem \ref{th3-2} for time fractional p-Laplace equation \eqref{eq3-11} using the "fundamental identity" (Lemma 3.1 in \citet{Za19}). In Sect. \ref{sec3-3} for initial boundary value problem with truncated potential \eqref{eqq29} we prove some useful appriory estimates. Finally, in Sect. \ref{sec4} the main results Theorem \ref{th1-1} and Theorem \ref{th1-2} are proved.
\section{Time fractional ordinary differential equations}
\label{sec2}
\subsection{Comparison principle}
\label{sec3-1}
 Consider the ordinary fractional differential equation
 \begin{equation}
 \label{eq3-1}
 D^\alpha_t w(t)=f(t,w(t)), \hbox{ for } t\in[0,T].
 \end{equation}
\begin{theorem}
\label{th3-1}
Suppose $u(t), v(t)$ are sub- and supersolutions of \eqref{eq3-1} and
$|u(t)|\leq M, \quad |v(t)|\leq M,  \hbox{ for } t\in[0,T]$,
\begin{equation}
\label{eq3-4}
D^\alpha_t u(t)\leq f(t,u(t)), \quad D^\alpha_t v(t)\geq f(t,v(t)), \hbox{ for } t\in[0,T],
\end{equation}
\begin{equation}
\label{eq3-5}
|f_\omega(t,\omega)|\leq K, \quad \hbox{ for } t\in[0,T], |\omega|\leq M.
\end{equation}
If $u(0)\leq v(0)$ then
\begin{equation}
\label{eq3-6}
u(t)\leq v(t),  \hbox{ for } t\in[0,T].
\end{equation}
\end{theorem}
\begin{proof}
Suppose by contradiction that \eqref{eq3-6} fails. Define the function
$$
z(t)=\left(u(t)-v(t)\right)_+=\left\{\begin{array}{lll} &u(t)-v(t)& \hbox{ for } u(t)-v(t)>0,\\ &0& \hbox{ for }  u(t)-v(t)\leq 0,\end{array}\right.
$$
Then there exists $T_0$, $0\leq T_0<T$, $T_0=\inf\{t; z(t)>0\}$ and function $z(t)=0$ for $t\in[0,T_0]$. Moreover, from \eqref{eq3-4} $D^\alpha_tz(t)$ satisfies the inequality
\begin{equation}
\label{eq3-7}
D^\alpha_t z(t)\leq f(t,u(t))-f(t,v(t))  \hbox{ for } t\in[0,T].
\end{equation}
Since
$$
f(t,u(t))-f(t,v(t))=\left(u(t)-v(t)\right)\int_0^1\frac{\partial f}{\partial\omega}\left(t,(1-\theta)v(t)+\theta u(t)\right)d\theta,
$$
and from \eqref{eq3-5} and $|(1-\theta)v(t)+\theta u(t)|\leq M$, $|f(t,u(t))-f(t,v(t))|\leq Kz(t)$ for $t\in[0,T]$ inequality \eqref{eq3-7} becomes
$$
D^\alpha_tz(t)\leq K z(t), z(0)=0.
$$
Thus
\begin{equation}
\label{eq3-8}
D^\alpha_t z(t)=\frac{1}{\Gamma(1-\alpha)}\frac{\partial}{\partial t}\int_0^t(t-s)^{-\alpha}z(s)ds\leq Kz(t), \hbox{ for } t\in[0,T],
\end{equation}
and integrating \eqref{eq3-8} from $0$ to $T_1\in(T_0,T]$ we have
\begin{equation}
\label{eq3-9}
\begin{array}{l}
\frac{1}{\Gamma(1-\alpha)}\int_0^{T_1}(T_1-s)^{-\alpha}z(s)ds-\frac{1}{\Gamma(1-\alpha)}\lim_{t\rightarrow0}\int_0^t(t-s)^{-\alpha}z(s)ds
\\[1pt]
\\
\leq K\int_0^{T_1}z(s)ds.
\end{array}
\end{equation}
From
$$
\int_0^t(t-s)^{-\alpha}z(s)ds\leq M\int_0^t(t-s)^{-\alpha}ds=M\frac{t^{1-\alpha}}{1-\alpha},
$$
it follows that $\lim_{t\rightarrow0}\int_0^t(t-s)^{-\alpha}z(s)ds=0$ and from \eqref{eq3-9}
$$
\begin{array}{l}
\int_{T_0}^{T_1}(T_1-s)^{-\alpha}z(s)ds=\int_0^{T_1}(T_1-s)^{-\alpha}z(s)ds\leq K\Gamma(1-\alpha)\int_0^{T_1}z(s)ds
\\[1pt]
\\
= K\Gamma(1-\alpha)\int_0^{T_1}z(s)ds.
\end{array}
$$
Thus
$$
(T_1-T_0)^{-\alpha}\int_{T_0}^{T_1}z(s)ds\leq K\Gamma(1-\alpha)\int_{T_0}^{T_1}z(s)ds,
$$
and for every $T_1<\min\left\{T, T_0+\left(K\Gamma(1-\alpha)\right)^{-\frac{1}{\alpha}}\right\}$ we get
$$
0<\left[(T_1-T_0)^{-\alpha}-K\Gamma(1-\alpha)\right]\int_{T_0}^{T_1}z(s)ds\leq0,
$$
i.e. $z(t)=0$ for $t\in[T_0,T_1]$ which contradict the choice of $T_0$.
\end{proof}
\subsection{Blow-up solutions}
\label{sec2-1}
Following \citet{SL24} we consider initial value problem for  ordinary fractional differential equation
\begin{equation}
\label{eqq1}
\left\{\begin{array}{ll}
&D^\alpha u(t)=u^p(t), \alpha\in (0,1), p>1, \hbox{ for } t>0,
\\
&u(0)=u_0>0
\end{array}
\right.
\end{equation}
The solution of \eqref{eqq1} is equivalent to the solution of the integral equation, see \citet{KST06}
\begin{equation}
\label{eqq2}
u(t)=u(0)+\frac{1}{\Gamma(\alpha)}\int_0^t(t-s)^{\alpha-1}u^p(s)ds,  \hbox{ for } t\geq0,
\end{equation}
Let us recall the comparison principle for integral equation \eqref{eqq2}.
\begin{theorem}[Lemma 2.1 in \citet{LV08}]
\label{th1}
Let $v(t), w(t)\in C([0,T])$, $F\in C([0,T]\times \mathbb{R})$ and let
one of the following inequalities being strict.
$$
\begin{array}{ll}
&v(t)\leq v(0)+\frac{1}{\Gamma(\alpha)}\int_0^t(t-s)^{\alpha-1}F(s,v(s))ds,
\\[1pt]
\\
&w(t)\leq w(0)+\frac{1}{\Gamma(\alpha)}\int_0^t(t-s)^{\alpha-1}F(s,w(s))ds.
\end{array}
$$
Suppose $F(t,x)$ is non decreasing in $x$ for every $t$ and $v(0)<w(0)$.

Then $v(t)<w(t)$ for $0\leq t\leq T$.
\end{theorem}

The following theorem that we will use in the proof of Theorem \ref{th1-2} gives an existence of a blow-up solution of problem \eqref{eqq1} and improves the result in  \citet{SL24}.
\begin{theorem}
\label{th2-3}
Suppose $u_0$ is strictly positive, $p>1$ and $\alpha\in(0,1)$. Then the solution of \eqref{eqq1} blows-up for finite time $T_m<\infty$.
\end{theorem}
\begin{proof}
For a  constant $\delta>0$, which will be chosen later on, we consider the problem.
\begin{equation}
\label{eqq3}
\left\{ \begin{array}{ll}
&w'(t)=\frac{w^p(t)}{\Gamma(\alpha)(t+\delta)^{p(1-\alpha)}}, \hbox{ for }t>0,
\\
&w(0)=\frac{1}{2}\delta^{1-\alpha}u_0.
\end{array}\right.
\end{equation}
Integrating \eqref{eqq3} we get
\begin{equation}
\label{eqq4}
w(t)=w(0)+\frac{1}{\Gamma(\alpha)}\int_0^t\frac{w^p(s)}{(s+\delta)^{p(1-\alpha)}}ds, \hbox{ for }  t\geq0,
\end{equation}
as well as
$$
\begin{array}{l}
\frac{dw}{w^p}=\frac{dt}{\Gamma(\alpha)((t+\delta)^{p(1-\alpha)})},
\\
\int_0^t\frac{dw}{w^p}=\int_0^t\frac{1}{\Gamma(\alpha)((s+\delta)^{p(1-\alpha)})}ds, \hbox{and}
\\
\frac{1}{1-p}\left[w^{1-p}(t)-w^{1-p}(0)\right]
\\
=\frac{1}{\Gamma(\alpha)(1-p(1-\alpha))}\left[(t+\delta)^{1-p(1-\alpha)}-\delta^{1-p(1-\alpha)}\right]
\end{array}
$$
Finally,
$$
\left\{\begin{array}{ll}
&w^{1-p}=F(t,\delta), \hbox{ for } p(1-\alpha)-1\neq0,
\\
& w^{1-p}=F_0(t,\delta), \hbox{ for } p(1-\alpha)-1=0,
\end{array}\right.
$$
where
$$
\begin{array}{ll}
&F(t,\delta)=w^{1-p}(0)+\frac{(p-1)\delta^{1-p(1-\alpha)}}{\Gamma(\alpha)[1-p(1-\alpha)]}-\frac{(p-1)(t+\delta)^{1-p(1-\alpha)}}{\Gamma(\alpha)[1-p(1-\alpha)]},
\\
&F_0(t,\delta)=w^{1-p}(0)-\frac{p-1}{\Gamma(\alpha)}\ln\frac{t+\delta}{\delta}.
\end{array}
$$
From \eqref{eqq2} the function $v(t)=(t+\delta)^{1-\alpha}u(t)$ satisfies the equality
\begin{equation}
\label{eqq6}
v(t)=(t+\delta)^{1-\alpha}u_0+\frac{(t+\delta)^{1-\alpha}}{\Gamma(\alpha)}\int_0^t\frac{(t-s)^{\alpha-1}v^p(s)}{(s+\delta)^{p(1-\alpha)}}ds.
\end{equation}
Since $0\leq t-s\leq t<t+\delta$ for $0\leq s\leq t$, we get from \eqref{eqq4}, \eqref{eqq6} the inequality
$$
\begin{array}{ll}
v(t)>&\delta^{1-\alpha}u_0+\frac{(t+\delta)^{1-\alpha}}{\Gamma(\alpha)}\int_0^t\frac{(t-s)^{\alpha-1}v^p(s)}{(s+\delta)^{p(1-\alpha)}}ds
\\
&=2w(0)+\frac{1}{\Gamma(\alpha)}\int_0^t\frac{v^p(s)}{(s+\delta)^{p(1-\alpha)}}ds.
\end{array}
$$
We will chose $\delta$ in order to prove the existence of blow-up solution of \eqref{eqq1}.

Let us consider the cases
\begin{itemize}
\item[i)] $1-p(1-\alpha)\geq0$,
\begin{itemize}
\item[i1)] $1-p(1-\alpha)>0$,
\item[i2)] $1-p(1-\alpha)=0$,
\end{itemize}
\item[ii)] $1-p(1-\alpha)<0$.
\end{itemize}
We will show that in both cases there exists $T_m<\infty$ such that $\lim_{t\rightarrow T_m,t<T_m}w(t)=\infty$.

In the case i1) and i2) we chose $\delta=1$.

In the case ii) denote $A=\frac{p-1}{\Gamma(\alpha)[p(1-\alpha)-1]}$ and we chose

$\delta=\left[\frac{3}{2}Aw^{p-1}(0)\right]^{\frac{1}{p(1-\alpha)-1}}>0$.

Indeed,  in case i) we have $F(0,1)=F_0(0,1)=w^{1-p}(0)>0$ and $\lim_{t\rightarrow\infty} F(t,1)=-\infty$ as well as  $\lim_{t\rightarrow\infty} F_0(t,1)=-\infty$. So there exists
$T_m<\infty$ such that $\lim_{t\rightarrow T_m,t<T_m}w^{1-p}(t)=0$ because $p>1$ and hence $\lim_{t\rightarrow T_m,t<T_m}w(t)=\infty$.

In the case i1) $F(T_m,1)=0$ means
$$
w^{1-p}(0)+\frac{p-1}{\Gamma(\alpha)[1-p(1-\alpha)]}=\frac{(p-1)(T_m+1)^{1-p(1-\alpha)}}{\Gamma(\alpha)[1-p(1-\alpha)]},
$$
$$
(T_m+1)^{1-p(1-\alpha)}=w^{1-p}(0)\frac{\Gamma(\alpha)[1-p(1-\alpha)]}{p-1}+1,
$$
and
\begin{equation}
\label{eqq7}
T_m=\left(w^{1-p}(0)\frac{\Gamma(\alpha)[1-p(1-\alpha)]}{p-1}-1\right)^{\frac{1}{p(1-\alpha)-1}}-1.
\end{equation}
Formulae \eqref{eqq7} is the same as in \citet{SL24} for the case i1)

In the case i2) solving the equation $F_0(T_m,1)=0$ we obtain
$$
T_m=e^{w^{1-p}(0)\frac{\Gamma(\alpha)}{p-1}}-1.
$$
In the case ii) we have that $F(T_m,\delta)=0$ if
\begin{equation}
\label{eqq9}
\frac{1}{w^{p-1}(0)}-\frac{A}{\delta^{p(1-\alpha)-1}}=\frac{A}{(T_m+\delta)^{p(1-\alpha)-1}}.
\end{equation}
With $\delta=\left(\frac{3Aw^{p-1}(0)}{2}\right)^{\frac{1}{p(1-\alpha)-1}}$, we get $\frac{1}{w^{p-1}(0)}=\frac{3A}{2\delta^{p(1-\alpha)-1}}$ and from \eqref{eqq9}
$$
\frac{A}{2\delta^{p(1-\alpha)-1}}=\frac{A}{(T_m+\delta)^{p(1-\alpha)-1}}.
$$
then $T_m=\left(2^{\frac{1}{p(1-\alpha)-1}}-1\right)\delta$ and finally
$$
T_m=\left(2^{\frac{1}{p(1-\alpha)-1}}-1\right)\left(\frac{3(p-1)w^{p-1}(0)}{2\Gamma(\alpha)[p(1-\alpha)-1]}\right)^{\frac{1}{p(1-\alpha)-1}},
$$
so $\lim_{t\rightarrow T_m, t<T_m} w(t)=\infty$.

From the comparison principle, Theorem \ref{th1}, for \eqref{eqq4} and \eqref{eqq7} it follows that
$$
v(t)>w(t), \hbox{ for }t\in[0,T_m),
$$
or equivalently
$$
u(t)\geq(t+\delta)^{\alpha-1}w(t),  \hbox{ for } t\in[0,T_m).
$$
Hence, from what is proved above in both cases i) and ii) we obtain
$$
\lim_{t\rightarrow T_m,t<T_m}u(t)=\infty.
$$

\end{proof}

\section{Time fractional parabolic problems}
\label{sec3}
\subsection{Comparison principle}
\label{sec3-2}
We consider the problem
\begin{equation}
\label{eq3-11}
D^\alpha_t w -\Delta_pw=f(x)|w|^{p-2}w, \quad \hbox{ for } t\in[0,T], x\in\Omega,
\end{equation}
where $\Omega$ is a bounded domain in $\mathbb{R}^n$ and
\begin{equation}
\label{eq3-12}
\alpha\in(0,1),\quad p>2, \hbox{ and } |f(x)|\leq C, \hbox{ for } x\in\Omega.
\end{equation}
\begin{theorem}[Comparison principle]
\label{th3-2}
Suppose
$$
u(t,x), v(t,x)\in C^0\left([0,T]; L^2(\Omega)\right)\cap L^p\left([0,T]; W_0^{1,p}\right),
$$
satisfy the inequality
$$
D^\alpha_tu-\Delta_pu-f(x)|u|^{p-2}u\leq \partial^\alpha_tv-\Delta_pv-f(x)|v|^{p-2}v, \quad \hbox{ for } t\in[0,T], x\in\Omega.
$$

If \eqref{eq3-12} holds and $u(0,x)\leq v(0,x)$ for $x\in \Omega$ then
\begin{equation}
\label{eq3-14}
u(t,x)\leq v(t,x)  \quad \hbox{ for } t\in[0,T], x\in\Omega.
\end{equation}
\end{theorem}
\begin{proof}
Suppose by contradiction that \eqref{eq3-14} fails. Then the function
$$
z(t,x)=\left(u(t,x)-v(t,x)\right)_+=\left\{\begin{array}{lll} &u(t,x)-v(t,x)& \hbox{ for } u(t,x)-v(t,x)>0,\\ &0& \hbox{ for }  u(t,x)-v(t,x)\leq 0,\end{array}\right.,
$$
satisfies the inequality
\begin{equation}
\label{eq3-15}
D^\alpha_tz-\Delta_pu+\Delta_pv-f(x)\left(|u|^{p-2}u-|v|^{p-2}v\right)\leq 0,
\end{equation}
for $(t,x)\in Q_+=\left\{(t,x); t\in[0,T], x\in\Omega_+(t)\right\}$ and ,
 $\Omega_+(t)=\left\{x\in \Omega, u(t,x)>v(t,x)\right\}$.
Let us recall the so called fundamental identity, proved in the next Lemma.
\begin{lemma}[Lemma 4.1 in \citet{Za19}]
Let $T>0$, $J=(0,T)$ and $U$ be an open subset of $\mathbb{R}$. Let further $k\in H^1_1(J)$, $H\in C^1(U)$, and $u\in L^1(J)$ with $u(t)\in U$  for a.a. $t\in J$.Suppose that  the functions $H(u),H'(u)u$ and $H'(u)(k\ast u)$ belong to $L^1(J)$. Then we have for a.a. $t\in J$ the equality
\begin{equation}
\label{eq3-16}
\begin{array}{l}
H'(u(t))\frac{1}{\Gamma(1-\alpha)}\frac{\partial}{\partial t}\int_0^t(t-s)^{-\alpha}(k\ast u)(s)ds
\\
=\frac{1}{\Gamma(1-\alpha)}\frac{\partial}{\partial t}\int_0^t(t-s)^{-\alpha}(k\ast H(u))(s)ds
+\left(-H(u(t))+H'(u(t))\right)k(t)
\\
+\int_0^t\left(H(u(t-s))-H(u(t))-H'(u(t))[u(t-s)-u(t)]\right)[-k(s)]ds.
\end{array}
\end{equation}
\end{lemma}
From the fundamental identity \eqref{eq3-16} for $k=g_{1-\alpha}(t)=\frac{t^{-\alpha}}{\Gamma(1-\alpha)}$ and $H(z)=z^2(t,x)$ we get
\begin{equation}
\label{eq3-17}
\begin{array}{l}
2z(t,x)\frac{1}{\Gamma(1-\alpha)}\frac{\partial}{\partial t}\int_0^t(t-s)^{-\alpha}z(s,x)ds
\\
=\frac{1}{\Gamma(1-\alpha)}\frac{\partial}{\partial t}\int_0^t(t-s)^{-\alpha}z^2(s,x)ds+\left[-z^2(t,x)+2z^2(t,x)\right]g_{1-\alpha}(t)
\\
-\int_0^t\left[z^2(t-s,x)-z^2(t,x)-2z(t,x)\left(z(t-s,x)-z(t,x)\right)\right]\frac{d}{ds}g_{1-\alpha}(s)ds
\\
=\frac{\partial}{\partial t}\int_0^t(t-s)^{-\alpha}z^2(s,x)ds+\frac{t^{-\alpha}}{\Gamma(1-\alpha)}z^2(t,x)
\\
+\frac{\alpha}{\Gamma(1-\alpha)}\int_0^ts^{-\alpha-1}\left[z(t-s,x)-z(t,x)\right]^2ds.
\end{array}
\end{equation}
Multiplying \eqref{eq3-15} with $z(t,x)$ and integrating in $\Omega$ we have
$$
\begin{array}{l}
\int_{\Omega_+(t)}z(t,x)\frac{\partial}{\partial t}\int_0^t(t-s)^{-\alpha}z^2(s,x)dsdx
\\
+\int_{\Omega_+(t)}\nabla z\left[|\nabla u|^{p-2}\nabla u-|\nabla v|^{p-2}\nabla v\right]dx
\\
-\int_{\Omega_+(t)}f(x)\left[|u|^{p-2}u-|v|^{p-2}v\right]z(t,x)dx\leq0
\end{array}
$$
Simple computation give us
$$
\begin{array}{l}
\left[|\nabla u|^{p-2}\nabla u-|\nabla v|^{p-2}\nabla v\right]\nabla z
=\frac{1}{2}\left[|\nabla v|^{p-2}+|\nabla u|^{p-2}\right]|\nabla(u-v)|^2
\\
+\frac{1}{2}\left[|\nabla u|^{p-2}+|\nabla v|^{p-2}\right]\left[|\nabla u|^2-|\nabla v|^2\right]\geq0,
\end{array}
$$
$$
0\leq\left( |u|^{p-2}u-|v|^{p-2}v\right) z
\leq(p-1)(v-u)^2\left(|v|+|u|\right)^{p-2},
$$
and hence
\begin{equation}
\label{eq3-21}
\begin{array}{l}
0\leq\int_{\Omega_+(t)}\left[|u(t,x|^{p-2}u(t,x)-|v(t,x)|^{p-2}v(t,x)\right]z(t,x)dx
\\
\leq(p-1)\left(\int_{\Omega_+(t)}|z(t,x)|^pdx\right)^\frac{2}{p}\left(\int_{\Omega_+(t)}(|v|^p+|u|^p)dx\right)^\frac{p-2}{p}
\\
\leq(p-1)2^{\frac{(p-1)(p-2)}{p}}\left(\int_{\Omega_+(t)}(|v|^p+|u|^p)dx\right)^\frac{p-2}{p}\left(\int_{\Omega_+(t)}|z(t,x)|^pdx\right)^\frac{2}{p}.
\end{array}
\end{equation}
If
$$
\sup_{t\in[0,T]}\int_\Omega|u(t,x)|^pdx\leq M, \quad \sup_{t\in[0,T]}\int_\Omega|v(t,x)|^pdx\leq M,
$$
then \eqref{eq3-21} becomes
\begin{equation}
\label{eq3-23}
\begin{array}{l}
\int_{\Omega_+(t)}\left[|u|^{p-2}u-|v|^{p-2}v\right]z(t,x)dx
\\
\leq(p-1)2^{\frac{(p-1)(p-2)}{p}}(2M)^{\frac{p-2}{p}}\left(\int_{\Omega_+(t)}|z(t,x)|^pdx\right)^{\frac{2}{p}}
\\
=(p-1)2^{p-2}M^{\frac{p-2}{p}}\left(\int_{\Omega_+(t)}|z(t,x)|^pdx\right)^{\frac{2}{p}}.
\end{array}
\end{equation}

From \eqref{eq3-17} - \eqref{eq3-23} and $z(0,x)\equiv0$ we get the inequalities
\begin{equation}
\label{eq3-24}
\begin{array}{l}
0\geq\frac{1}{2}D^\alpha_t\int_{\Omega_+(t)}z^2(t,x)dx-\frac{1}{2}D^\alpha_t\int_{\Omega_+(0)}z^2(0,x)dx
\\
+\frac{1}{2\Gamma(1-\alpha)}\int_{\Omega_+(t)}t^{-\alpha}z^2(t,x)dx
\\
+\frac{\alpha}{2\Gamma(1-\alpha)}\int_{\Omega_+(t)}\int_0^ts^{-\alpha-1}\left[z(t-s,x)-z(t,x)\right]^2dsdx
\\
+\int_{\Omega_+(t)}\left[|\nabla u|^{p-2}\nabla u-|\nabla v|^{p-2}\nabla v\right]\nabla zdx
\\
-C\int_{\Omega_+(t)}\left[|u|^{p-2}u-|v|^{p-2}v\right]zdx
\\
\geq \frac{1}{2}D^\alpha_t\int_{\Omega_+(t)}z^2(t,x)dx+\frac{1}{2\Gamma(1-\alpha)}\int_{\Omega_+(t)}t^{-\alpha}z^2(t,x)dx
\\
+\frac{\alpha}{2\Gamma(1-\alpha)}\int_{\Omega_+(t)}\int_0^ts^{-\alpha-1}\left[z(t-s,x)-z(t,x)\right]^2dsdx
\\
+\frac{1}{2}\int_{\Omega_+(t)}\left[\left(|\nabla u|^{p-2}+|\nabla v|^{p-2}\right)|\nabla z|^2\right.
\\
\left.+\left(|\nabla u|^{p-2}-|\nabla v|^{p-2}\right)\left(|\nabla u|^2-|\nabla v|^2\right)\right]dx
\\
-C(p-1)2^{p-2}M^{\frac{p-2}{p}}\left(\int_{\Omega_+(t)}|z(t,x)|^pdx\right)^{\frac{2}{p}}
\\
\geq\frac{1}{2}D^\alpha_t\int_{\Omega_+(t)}z^2(t,x)dx-K_1\left(\int_{\Omega_+(t)}|z(t,x)|^pdx\right)^{\frac{2}{p}},
\end{array}
\end{equation}
where
\begin{equation}
\label{eq3-25}
K_1=C(p-1)2^{p-2}M^{\frac{p-2}{p}}.
\end{equation}
Since
\begin{equation}
\label{eq3-26}
\begin{array}{l}
\left(\int_{\Omega_+(t)}|z(t,x)|^pdx\right)^{\frac{2}{p}}=\int_{\Omega_+(t)}z^2(t,x)dx\left(\int_{\Omega_+(t)}|z(t,x)|^pdx\right)^{\frac{2}{p}}
\\
\times\left(\int_{\Omega_+(t)}|z(t,x)|^2dx\right)^{-1}=\theta\int_{\Omega_+(t)}|z(t,x)|^2dx
\end{array}
\end{equation}
where $\theta(z)=\left(\int_{\Omega_+(t)}|z(t,x)|^pdx\right)^{\frac{2}{p}}\left(\int_{\Omega_+(t)}|z(t,x)|^2dx\right)^{-1}$, is a homogeneous function of 0-th order, i.e. $\theta(\lambda z)=\theta(z)$ for every $\lambda\neq0$ we have
$$
0\leq\theta(z)\leq C_1=|\Omega_+(\delta)|^{\frac{p-2}{p}}.
$$
Here $C_1=\theta(\delta)$ for some $\delta\in(0,T]$ such that $\Omega_+(\delta)\neq0$.

Thus, from \eqref{eq3-25}, \eqref{eq3-26}, inequality \eqref{eq3-24} becomes for $w(t)=\int_{\Omega_+(t)}|z(t,x)|^2dx$,
$$
D^\alpha_tw(t)\leq2K_1C_1w(t) \hbox{ for } t\in[0,T], w(0)=0.
$$
Since
\begin{equation}
\label{eq3-29}
0\leq\int_{\Omega_+(t)}z^2(t,x)dx\leq\left(\int_{\Omega_+(t)}|z|^pdx\right)^{\frac{2}{p}}\left(\int_{\Omega_+(t)}1dx\right)^{\frac{p-2}{p}}\leq|\Omega_+(t)|^{\frac{p-2}{p}}M,
\end{equation}
and \eqref{eq3-12} holds, from Theorem \ref{th3-1} for $u(t)=w(t), v(t)=0$, it follows that $w(t)\leq0$, i.e. $z(t,x)=0$ for $t\in[0,T]$, $x\in\Omega$ which contradicts our assumption.
\end{proof}
\subsection{Apriori estimates}
\label{sec3-3}
In this section we consider equation \eqref{eq3-11} with initial and boundary data
\begin{equation}
\label{eqq29}
\left\{\begin{array}{ll}
&D^\alpha_t u_N-\Delta_pu=f_N(x)|u_N|^{p-2}u_N, \quad \hbox{ for } t\in[0,T], x\in\Omega,
\\
&u_N(0,x)=u_0(x), \hbox{ for } x\in \Omega,
\\
&u_N(t,x)=0, \hbox{ for } t\in[0,T], x\in \partial\Omega,
\end{array}\right.
\end{equation}
where $\alpha\in(0,1)$, $n>p>2$,  $f_N(x)=\mu W_N(x)=\mu \min\{N, W(x)\}, N \in \mathbb{N}$.
Here $W(x)$ is defined in \eqref{eq3} and $\mu<\Lambda(n,p)$.
\begin{theorem}
\label{th3}
There exist constants $A_1$, $A_2$ and the weak solution $u_N(t,x)$ of \eqref{eqq29} satisfies the estimates
\begin{equation}
\label{eq3-40a}
\int_0^T\int_\Omega |\nabla u_N(t,x)|^pdtdx\leq A_1, \quad \int_0^T\int_\Omega |u_N(t,x)|^pdtdx\leq A_2.
\end{equation}
\end{theorem}
\begin{proof}
From \eqref{eq3-11}, \eqref{eq3-12} and \eqref{eq3-17} for $z=u_N(t,x)$ we get
\begin{equation}
\label{eq3-33}
\begin{array}{l}
0=2u_N(t,x)\frac{\partial}{\partial t}\frac{1}{\Gamma(1-\alpha)}\int_0^t(t-s)^{-\alpha}u_N(s,x)ds
\\
-2u_N(t,x)\frac{\partial}{\partial t}\frac{1}{\Gamma(1-\alpha)}\int_0^t(t-s)^{-\alpha}u_0(x)ds-2u(t,x)\Delta_pu_N(t,x)
\\
+2u_N(t,x)D^\alpha_tu_0(x)-2u_N(t,x)\mu W_N(x)|u_N(t,x|^{p-2}u_N(t,x)
\\
=D^\alpha_tu^2_N(t,x)+2u^2_N(t,x)\frac{\partial}{\partial t}\frac{1}{\Gamma(1-\alpha)}\int_0^t(t-s)^{-\alpha}u^2_0(x)ds
\\
+\frac{t^{-\alpha}}{\Gamma(1-\alpha)}u^2_N(t,x)+\frac{\alpha}{\Gamma(1-\alpha)}\int_0^ts^{-\alpha-1}\left[u_N(t-s,x)-u_N(t,x)\right)^2ds
\\
-2u_N(t,x)\frac{\partial}{\partial t}\frac{1}{\Gamma(1-\alpha)}\int_0^t(t-s)^{-\alpha}u_0(x)ds
\\
-2u_N(t,x)\Delta_pu_N(t,x)-2u(t,x)\mu W_N(x)|u_N(t,x|^{p-2}u_N(t,x).
\end{array}
\end{equation}
Integrating \eqref{eq3-33} in $\Omega$ we obtain
\begin{equation}
\label{eq3-34}
\begin{array}{l}
0=D_t^\alpha\int_\Omega u^2_N(t,x)dx+\frac{t^{-\alpha}}{\Gamma(1-\alpha)}\int_\Omega u^2_N(t,x)dx
\\
+\frac{\alpha}{\Gamma(1-\alpha)}\int_0^t\int_\Omega s^{-\alpha-1}\left[u_N(t-s,x)-u_N(t,x)\right]^2dsdx
\\
+2\int_\Omega|\nabla u_N(t,x)|^pdx-2\mu\int_\Omega W_N(x)|u_N(t,x)|^pdx
\\
+\frac{\partial}{\partial t}\frac{1}{\Gamma(1-\alpha)}\int_\Omega\int_0^t(t-s)^{-\alpha}u^2_0(x)dsdx
\\
-2u_N(t,x)\frac{\partial}{\partial t}\frac{1}{\Gamma(1-\alpha)}\int_\Omega\int_0^t(t-s)^{-\alpha}u_0(x)dsdx
\end{array}
\end{equation}
Since $\frac{\partial}{\partial t}\frac{1}{\Gamma(1-\alpha)}\int_0^t(t-s)^{-\alpha}u^j_0(x)ds=\frac{t^{-\alpha}}{\Gamma(1-\alpha)}u_0^j(x)$, $j=1,2$  and
$$
\begin{array}{l}
2\left|\int_\Omega\frac{u_N(t,x)}{\Gamma(1-\alpha)}\frac{\partial}{\partial t}\int_0^t(t-s)^{-\alpha}u_0(x)dxds\right|= \frac{2t^{-\alpha}}{\Gamma(1-\alpha)}\left|\int_\Omega u_N(t,x)u_0(x)dx\right|
\\
\leq\frac{t^{-\alpha}}{2\Gamma(1-\alpha)}\int_\Omega u^2_N(t,x)dx+\frac{2t^{-\alpha}}{\Gamma(1-\alpha)}\int_\Omega u^2_0(x)dx
\end{array}
$$
we get from \eqref{eq3-34} and Hardy inequalities \eqref{eq3a}
\begin{equation}
\label{eq3-35}
\begin{array}{l}
0\geq D^\alpha_t\int_\Omega u^2_N(t,x)dx+\frac{t^{-\alpha}}{2\Gamma(1-\alpha)}\int_\Omega u^2_N(t,x)dx+2\int_\Omega|\nabla u_N(t,x)|^pdx
\\
-2(\mu+\varepsilon)\Lambda^{-1}(n,p)\int_\Omega|\nabla u_N(t,x)|^pdx
\\
+2\varepsilon\int_\Omega W_N(x)|u_N(t,x)|^pdx-\frac{3t^{-\alpha}}{\Gamma(1-\alpha)}\int_\Omega u_0^2(x)dx,
\end{array}
\end{equation}
where
$$
0<\varepsilon<\min\left(\omega_0, \Lambda(n,p)-\mu\right), \quad \omega_0=\min_\Omega W(x)>0.
$$
Integrating \eqref{eq3-35} on $t$ from $0$ to $T $ we have
\begin{equation}
\label{eq3-37}
\begin{array}{l}
0\geq\int_0^T\frac{\partial}{\partial t}\int_0^t\frac{(t-s)^{-\alpha}}{\Gamma(1-\alpha)}\int_\Omega u^2_N(s,x)dsdxdt
\\
+\gamma\int_0^T\int_\Omega|\nabla u_N(t,x)|^pdxdt+2\varepsilon\omega_0\int_\Omega|u_N(t,x)|^pdxdt
\\
+\frac{1}{2\Gamma(1-\alpha)}\int_0^Tt^{-\alpha}\int_\Omega u^2_N(t,x)dxdt-\frac{3}{\Gamma(1-\alpha)}\int_0^Tt^{-\alpha}\int_\Omega u^2_0(x)dxdt,
\end{array}
\end{equation}
where $\gamma=2\left[1-(\mu+\varepsilon)C^{-1}\right]>0$.

From\eqref{eq3-29} it follows that
$$
0\leq\int_\Omega u^2_N(t,x)dx\leq M|\Omega|^\frac{p-2}{p}=C_2<\infty \hbox{ fof } t\in[0,T],
$$
and hence
$$
\begin{array}{l}
0\leq\lim_{t\rightarrow0}\int_0^t\frac{(t-s)^{1-\alpha}}{\Gamma(1-\alpha)}\int_\Omega u^2_N(t,x)dsdx
\\
\leq \frac{C_2}{\Gamma(1-\alpha)}\lim_{t\rightarrow0}\int_0^t(t-s)^{-\alpha}ds=\frac{C_2}{(1-\alpha)\Gamma(1-\alpha)}\lim_{t\rightarrow0}t^{1-\alpha}=0.
\end{array}
$$
Thus \eqref{eq3-37} becomes
$$
\begin{array}{l}
0\geq\int_0^T\frac{(T-s)^{1-\alpha}}{\Gamma(1-\alpha)}\int_\Omega u^2_N(s,x)dsdx+\gamma\int_0^T\int_\Omega|\nabla u_N(t,x)|^pdxdt
\\
+\frac{1}{2\Gamma(1-\alpha)}\int_0^Tt^{-\alpha}\int_\Omega u^2_N(t,x)dxdt+2\varepsilon\omega_0\int_0^T\int_\Omega |u_N(t,x)|^pdxdt
\\
-\frac{T^{1-\alpha}}{(1-\alpha)\Gamma(1-\alpha)}\int_\Omega u_0^2(x)dx
\\
\geq\gamma\int_0^T\int_\Omega|\nabla u_N(t,x)|^pdxdt+2\varepsilon\omega_0\int_0^T\int_\Omega |u_N(t,x)|^pdxdt
\\
-\frac{3T^{1-\alpha}}{(1-\alpha)\Gamma(1-\alpha)}\int_\Omega u_0^2(x)dx.
\end{array}
$$
and hence we obtain \eqref{eq3-40a} with $A_1=\frac{C_3}{\gamma}$, $A_2=\frac{C_3}{2\varepsilon\omega_0\Gamma(1-\alpha)}$
and $C_3=\frac{3T^{1-\alpha}}{1-\alpha}\int_\Omega u_0^2(x)dx$.
\end{proof}
\section{Proof of the main results}
\label{sec4}
In this section we prove global boundedness or finite time blow-up of the solutions in distributional sense of the singular fractional parabolic problem \eqref{eq1}
\begin{proof}[Proof of theorem \ref{th1-1}]
From the appriori estimate \eqref{eq3-40a} it follows that the solution $u_N(t,x)$ of problem \eqref{eq4} satisfies \eqref{eq4a}. After the limit $N\rightarrow\infty$ in \eqref{eq3-40a} by means of Theorem 4.1 in \citet{BM92} the limit solution $u(t,x)$ fulfilled \eqref{eq4a} and does not blow-up at $T_m$.
\end{proof}
\begin{proof}[Proof of theorem \ref{th1-2}]
We suppose by contradiction that $u(t,x)$ is globally defined for every $t>0$. Hence $u_N(t,x)$ is globally defined for every $t>0$. If $\lambda_N$ is the first eigenvalue of the problem
$$
\left\{\begin{array}{l}-\Delta_p v(x)=\lambda_NW_N(x)|v(x)|^{p-2}v(x)  \hbox{ for } x\in \Omega,\\ v(x)=0 \hbox{ for } x\in \partial\Omega \end{array}\right.,
$$
then from  Lemma 9.2 in \citet{KR22} we have
$$
\lambda_N\geq\Lambda(n,p), \quad \lim_{N\rightarrow\infty}\lambda_N=\Lambda(n,p),
$$
and problem
$$
\left\{\begin{array}{l} -\Delta_pX(x)-\mu W_N(x)|X(x)|^{p-2}X(x)=-X(x)  \hbox{ for } x\in \Omega,
\\
 X(x)=0 \hbox{ for } x\in \partial\Omega
\end{array}\right.,
$$
has a positive solution.

From Theorem \ref{th2-3} problem
$$
\left\{\begin{array}{l} D^\alpha_tT(t)=T^{p-1}(t), t>0,
\\T(0)=T_0>0 \end{array}\right.,
$$
has a positive solution which blows-up for finite time $T_{max}$. We chose $T_0=\varepsilon>0$ where $\varepsilon$ is small enough constant such that
$$
\varepsilon X(x)<u_0(x) \hbox{ for } x\in\Omega.
$$
Since $v(t,x)=T(t)X(x)$ is a positive subsolution of problem \eqref{eq4} for $t\in(0, T_{max})$ from the comparison principle Theorem \ref{th1} it follows that
$$
u_N(t,x)\geq T(t)X(x), \hbox{ for } x\in\Omega, t>0.
$$
Hence $u_N(t,x)$ blows-up at $t=T_{max}$ and the solution $u(t,x)=\lim_{N\rightarrow\infty}u_N(t,x)$ of \eqref{eq1} blows-up at $t=T_{max}$.
\end{proof}


\begin{flushleft}
Institute of Mathematics and  Informatics,\\ Bulgarian Academy
of Sciences \\ Acad. G. Bonchev str.,bl. 8\\
Sofia 1113, Bulgaria, \\ E-mail: kutev@math.bas.bg; rangelov@math.bas.bg,
\end{flushleft}

\end{document}